\documentclass[11pt, letterpaper,reqno]{amsart}
\usepackage{amsmath,amsthm,amsrefs,amssymb,enumitem,hyperref}
\usepackage[dvipsnames]{xcolor}
\usepackage[percent]{overpic}
\usepackage{float,graphicx}
\usepackage{caption}
\numberwithin{equation}{section}

\newtheorem{corollary}{Corollary}[section]
\newtheorem{lemma}{Lemma}[section]
\newtheorem{theorem}{Theorem}[section]

\newtheorem{definition}{Definition}[section]
\theoremstyle{definition}

\DeclareMathOperator{\D}{\mathbb{D}}
\DeclareMathOperator{\C}{\mathbb{C}}

\DeclareMathOperator{\Arg}{Arg}

\DeclareMathOperator{\dist}{dist}

\begin{document}
\title{Unbounded H\"{o}lder domains}

\author{Christina Karafyllia}  
\address{Department of Mathematics, University of Western Macedonia, Greece}
\email{xristinakrf@gmail.com}

\date{\today}

\begin{abstract} Motivated by the classical bounded H\"{o}lder domains, we introduce the notion of an unbounded simply connected H\"{o}lder domain. We prove analytic and  geometric characterizations of those domains with the aid of the spherical metric and the hyperbolic metric. We also study the relation of our definition to the definition of classical bounded H\"{o}lder domains. It turns out that unbounded H\"older domains form a natural class of domains for the study of the Hardy number (which determines the Hardy spaces to which the corresponding Riemann mapping belongs to). As an application of our characterizations, we prove a sharp bound for the Hardy number of an unbounded H\"{o}lder domain. 
\end{abstract}

\maketitle

\section{Introduction}\label{int}

Let $D$ be a bounded simply connected domain in the complex plane $\C$ and $f$ be a Riemann mapping from the unit disk $\D$ onto $D$. We say that $D$ is a H\"{o}lder domain, as introduced by Becker--Pommerenke \cite{BecPom}, if $f$ satisfies the H\"{o}lder condition. Specifically, for $0<\alpha\le 1$, $D$ is an $\alpha$-H\"{o}lder domain if there is a constant $K>0$ such that, for every $z_1,z_2\in \D$,
\begin{equation}\label{def1}
|f(z_1)-f(z_2)|\le K |z_1-z_2|^{\alpha}.
\end{equation}
The main result of Becker--Pommerenke shows that there is a geometric characterization of H\"{o}lder domains with the aid of the hyperbolic distance $h_D$ in $D$. Namely, if we fix $w_0 \in D$, then $D$ is an $\alpha$-H\"{o}lder domain if and only if there is a constant $C$ such that, for every $w\in D$,
\begin{equation}\label{def3}
h_D(w_0,w)\le C+\frac{1}{\alpha}\log \frac{\dist (w_0,\partial D)}{\dist (w,\partial D)}.
\end{equation}
Here, $\dist(z,A)$ denotes the Euclidean distance from a point $z\in \C$ to a set $A\subset \C$. This result was later improved by Smith--Stegenga in \cite{Smith2}. Moreover, N\"{a}kki--Palka established geometric conditions for a conformal mapping of a simply connected domain onto $\D$ (or its inverse) to be H\"older continuous, involving modulus \cite{NP}, hyperbolic distance \cite{NP2} or quasiconvexity \cites{NP3}.

Subsequently, the notion of H\"older domains has been extended to general domains. To define the property of a H\"{o}lder domain in a domain of arbitrary connectivity, we need the notion of the quasi-hyperbolic distance. Let $D$ be a proper subdomain of $\C$. The quasi-hyperbolic distance between two points $w_1,w_2\in D$ is defined by
\[k_D (w_1,w_2)=\inf_{\gamma} \int_{\gamma }\frac{|dw|}{\dist (w,\partial D)},\]
where the infimum is taken over all rectifiable paths $\gamma$ in $D$ connecting $w_1$ and $w_2$. It is well-known that $k_D$ is comparable to the hyperbolic distance $h_D$ when $\partial D$ is uniformly perfect (see \cites{BeaPom}). Now, we fix $w_0\in D$. Following Smith--Stegenga \cite{Smith1}, $D$ is a H\"{o}lder domain if there are constants $c_1\in \mathbb R$ and $c_2>0$ such that, for every $w\in D$,
\begin{align}\label{def_qh}
k_D(w_0,w)\le c_1+ c_2 \log\frac{\dist (w_0,\partial D)}{\dist (w,\partial D)}.
\end{align}
 As a consequence of that definition, H\"older domains must be bounded \cite{Smith1}.

H\"older domains appear naturally in complex dynamics \cite{GS}, in SLE \cite{RS}, and in random trees \cite{RS}. Although the domains appearing in these settings can be unbounded, their boundary is bounded; hence, by restricting to a bounded subdomain contained in a large ball, one can verify that the above H\"older domain condition is satisfied. On the other hand, such techniques cannot be applied in the case of an unbounded domain with \textit{unbounded boundary}. The purpose of the current work is to extend the definition of H\"older domains to this setting and develop the machinery for their study. Domains with unbounded boundary are very natural for the study of the Hardy number, which we discuss in the end of this section. 

Moreover, H\"older domains contain the well-studied class of John domains. John domains were introduced by F.\ John \cite{Jo} and studied by Martio--Sarvas \cite{MS}. Later, Nakki--V\"ais\"al\"a \cite{NV} introduced the notion of unbounded John domains and they observed that the most subtle case of such domains is that of unbounded boundary, as in our setting; see \cite{NV}*{\S 2.25}. 

We now give the definition of an unbounded simply connected H\"older domain. For this purpose, we substitute the Euclidean distance in \eqref{def1} with the spherical distance $\sigma$.

\begin{definition} Let $D\subsetneq \C$ be an unbounded simply connected domain and $f$ be a Riemann mapping from $\D$ onto $D$. We say that $D$ is an unbounded H\"{o}lder domain if there exist $\alpha\in (0,1]$ and $K>0$ such that, for every $z_1,z_2\in \D$,
\begin{equation}\label{definition}
\sigma (f(z_1),f(z_2))\le K|z_1-z_2|^{\alpha}.
\end{equation}
In this case, we say that $D$ is an unbounded $\alpha$-H\"older domain. 
\end{definition}

Henceforth, let $D\subsetneq \C$ be an unbounded simply connected domain and $f$ be a Riemann mapping from $\D$ onto $D$. Our first characterization involves the spherical derivative $f^\#$. 
 
\begin{theorem}
Let $0<\alpha\le 1$. Then $D$ is an unbounded $\alpha$-H\"{o}lder domain if and only if there is a constant $M>0$ such that, for every $z\in \D$,
\[f^{\#}(z)\le \frac{M}{(1-|z|)^{1-\alpha}}.\] 
\end{theorem}

It is a well-known theorem of Hardy--Littlewood (see \cite{Gol}*{p.~411}) that the same condition with the classical derivative $f'$ in the place of $f^\#$ characterizes the classical (bounded) H\"older domains. We provide the proof of the above theorem in Section \ref{defderivative}.

In Section \ref{defhyperbolic}, we establish a more geometric characterization in terms of the hyperbolic distance, which is the analogue of (\ref{def3}) for unbounded H\"older domains. We set $\dist_\sigma(z,A)$ be the spherical distance from a point $z$ to a set $A$. Fix a point $w_0\in D$.

\begin{theorem}
Let $0<\alpha\leq 1$. Then $D$ is an unbounded $\alpha$-H\"older domain if and only if there exists a constant $C$ such that, for every $w\in D$,
\[h_D(w_0,w)\le C+\frac{1}{\alpha}\log\frac{1}{\dist_\sigma(w,\partial D)}.\]
\end{theorem}

As an analogue to a result of Smith--Stegenga in \cite{Smith2} we also prove that $D$ is an unbounded $\alpha$-H\"older domain if and only if there is a constant $C$ such that, for every prime end $\tilde w$ of $D$ and every point $w$ on the hyperbolic geodesic $\Gamma$ from $w_0$ to $\tilde w$ in $D$,
\[h_D(w_0,w)\le C+\frac{1}{\alpha}\log\frac{1}{l_\sigma(\Gamma(w,\tilde w))}.\]
Here, $l_\sigma(\Gamma(w,\tilde w))$ denotes the spherical length of the subarc of $\Gamma$ joining $w$ to $\tilde{w}$. More details are given in Section \ref{defhyperbolic}.

In Section \ref{defholderbounded}, we complete the basic theory of unbounded H\"{o}lder domains, showing the relation of our definition to the definition of classical (bounded) H\"{o}lder domains. For this purpose, we set $r=\dist (w_0,\partial D)/2$ and we consider the function  
\[g(w)=\frac{r}{w-w_0},\]
for $w\in D$. If we restrict ourselves to the doubly connected bounded domain $D_0=g(D)\cap D(0,R)$, for an appropriate large $R>0$, then we prove the following result.

\begin{theorem} $D$ is an unbounded H\"{o}lder domain if and only if $D_0$ is a H\"{o}lder domain according to definition \eqref{def_qh}. 
\end{theorem}

We note that the constants appearing in the above statements depend quantitatively on each other and on the domain $D$. We formulate the results and the dependence of the various constants more precisely in the corresponding sections. The quantitative nature of the results poses some further difficulties in the proofs. 

Finally, as an application of our results, in Section \ref{application} we study the Hardy number of unbounded $\alpha$-H\"{o}lder domains and its relation to $\alpha$. The Hardy space ${H^p}\left( \mathbb{D} \right)$, where $p>0$, is defined to be the set of all holomorphic functions, $f$, on $\mathbb{D}$ that satisfy the condition 
\[\mathop {\sup }\limits_{0 < r < 1} \int_0^{2\pi } {{{| {f( {r{e^{i\theta }}} )} |}^p}d\theta  <  \infty }\]
(see \cite{Dur}). The fact that a function $f$ belongs to some space ${H^p}\left( \mathbb{D} \right)$ imposes a restriction on its growth and this restriction is stronger as $p$ increases. In other words, if $p>q$ then $H^p (\mathbb{D}) \subset H^q (\mathbb{D})$. In \cite{Han} Hansen studied the problem of determining the numbers $p$ for which $f \in {H^p}\left( \mathbb{D} \right)$ by studying $f\left( \mathbb{D} \right)$. For this purpose, he introduced a number which he called the Hardy number of a region. Since we study simply connected domains, we state the definition only in this case. Let $D\subsetneq \mathbb{C}$ be a simply connected domain and $f$ be a Riemann mapping from $\mathbb{D}$ onto $D$. The Hardy number of $D$ is defined by 
\[{\rm h} ( D ) = \sup \left\{ {p > 0:f \in {H^p}\left( \mathbb{D} \right)} \right\}.\]
This definition is independent of the choice of the Riemann mapping onto $D$. Since $D$ is simply connected, ${\rm h}(D)$ takes values in $[1/2, \infty]$. If $D$ is bounded then ${\rm h}(D)=\infty$ and thus we are interested in the case $D$ is unbounded. 

A classical problem in geometric function theory is to estimate the Hardy number of a plane domain. In this direction, Hansen \cite{Han} gave a lower bound for the Hardy number of an arbitrary plane domain and improved this bound for simply connected domains. Kim and Sugawa \cite{Kim} gave a description for the Hardy number of a plane domain in terms of harmonic measure and the current author established a characterization for the Hardy number of simply connected domains involving hyperbolic distance \cite{Kar}. However, we know more precise estimates for the Hardy number of certain types of domains such as starlike \cite{Han}, spiral-like \cite{Han2}, comb domains \cite{Karfin}, image regions of K{\oe}nigs mappings \cite{Co2}, etc.

In \cite{Kim} Kim and Sugawa studied the Hardy number of unbounded quasidisks, which furnish a subclass of unbounded H\"{o}lder domains, and gave sharp lower and upper bounds for this. Applying the theory developed before, we also establish a sharp upper bound for the Hardy number of an unbounded $\alpha$-H\"{o}lder domain in terms of $\alpha$.

\begin{theorem}\label{the14em} Let $D\subsetneq \C$ be an unbounded simply connected domain and $0<\alpha\leq 1$. If $D$ is an unbounded $\alpha$-H\"{o}lder domain, then ${\rm h}(D)\le 1/\alpha$ and this estimate is sharp.
\end{theorem} 

Finally, we note that an interesting property of the Hardy number of a simply connected domain $D$ is that it relates to the exit time of Brownian motion from $D$. In particular, if ${\widetilde h}(D)$ is the supremum of all $p>0$ for which the $p$-th moment of the exit time of Brownian motion is finite, then ${\widetilde h}(D)={\rm h}(D)/2$; see \cite{Bur}.

\section{preliminaries}

\subsection{Spherical distance} The {\textit{spherical length}} of a path $\gamma$ in $\C_{\infty}=\C\cup\{\infty\}$ is
\[l_{\sigma} (\gamma)=2\int_{\gamma}\frac{|dz|}{1+|z|^2}\]
and the \textit{spherical distance} between two points $z,w\in \C_{\infty}$ is
\[\sigma (z,w)=\inf_{\gamma} l_{\sigma}(\gamma),\]
where the infimum is taken over all the paths $\gamma$ joining $z$ to $w$. If $f$ is a meromorphic function then the \textit{spherical derivative} of $f$ is defined by
\[f^{\#}(z)=\frac{2|f'(z)|}{1+|f(z)|^2}.\]
If $\gamma$ is a curve, then the spherical length of the image curve $f\circ \gamma$ is
\[l_{\sigma}(f\circ \gamma)=\int_{\gamma} f^{\#}(z)|dz|.\]
The spherical length, the spherical distance and the spherical derivative are all invariant under rotations of the sphere.

The spherical distance between two points is comparable to their chordal distance. The {\textit{chordal distance}} between two points $z,w \in \C$ is given by
\[\chi (z,w)=\frac{2|z-w|}{\sqrt{1+|z|^2}\sqrt{1+|w|^2}}. \]
Moreover, $\chi (\infty,\infty)=0$ and $\chi (z,\infty)=2/\sqrt{1+|z|^2}$. The chordal distance is also invariant under rotations of the sphere. For every $z,w\in \C_{\infty}$,
\begin{equation}\label{sphecho}
\chi (z,w)\le {\sigma}(z,w)\le \frac{\pi}{2} \chi (z,w). 
\end{equation}

\subsection{Koebe's distortion theorem} Let $f$ be an analytic and univalent function in $\D$. We state the Koebe distortion theorem \cite[p.\ 9]{Pom}.

\begin{theorem}\label{koebedistortion} If $f$ maps $\D$ conformally into $\C$ then, for $z\in \D$,
\[|f'(0)|\frac{|z|}{(1+|z|)^2}\le |f(z)-f(0)|\le |f'(0)|\frac{|z|}{(1-|z|)^2}\]
and
\[|f'(0)| \frac{1-|z|}{(1+|z|)^3}\le|f'(z)|\le |f'(0)|\frac{1+|z|}{(1-|z|)^3}.\]
\end{theorem}

\begin{corollary}\label{koebedistance} If $f$ maps $\D$ conformally into $\C$ then, for $z\in \D$,
\[\frac{1}{4}(1-|z|)|f'(z)|\le \dist (f(z),\partial f(\D))\le 2(1-|z|)|f'(z)|. \]
\end{corollary}

We can also apply the above results to a conformal mapping defined on an arbitrary open disk $D(z_0,r)\subset \C$ using an appropriate linear transformation from $\D$ onto $D(z_0,r)$. 

\subsection{Geometric properties in the disk} Let $x,y$ be two points on a circle. If $C_{xy}$ is the smallest arc of the circle joining $x$ to $y$, then
\begin{equation}\label{archord}
|x-y| \le l(C_{xy})\le\frac{\pi}{2}|x-y|,
\end{equation}
where $l(A)$ denotes the Euclidean length of a curve $A$. Now, let $\Gamma$ be a hyperbolic geodesic in $\D$ joining two points on $\partial \D$. Let $z\in \Gamma $ and let $y$ be the endpoint of $\Gamma$ which is closer to $z$. If $\Gamma_{zy}$ denotes the subarc of $\Gamma$ joining $z$ to $y$, then
\begin{equation}\label{geodesic}  
1-|z|\le l(\Gamma_{zy}) \le \frac{\pi}{2} (1-|z|).
\end{equation}
This can be seen applying a M\"{o}bius transformation from $\D$ onto the upper half-plane and then using elementary calculus and geometric arguments. 

\section{Unbounded H\"{o}lder domains and spherical derivative}\label{defderivative}

\begin{theorem}\label{hoderivative} Let $D\subsetneq \C$ be an unbounded simply connected domain and $f$ be a Riemann mapping from $\D$ onto $D$. If there are constants $K>0$ and $0<\alpha\le 1$ such that, for every $z_1,z_2\in \D$,
\begin{equation}\label{defcond}
\sigma (f(z_1),f(z_2))\le K|z_1-z_2|^{\alpha},
\end{equation}
then, for every $z\in \D$,
\[f^{\#}(z)\le \frac{M}{(1-|z|)^{1-\alpha}},\]
where $M=90K(1+\dist(0,\C\backslash D))$. Conversely, if there exist constants $M>0$ and $0<\alpha\le 1$ such that, for every $z\in \D$,
\[f^{\#}(z)\le \frac{M}{(1-|z|)^{1-\alpha}},\]
then  (\ref{defcond}) is satisfied for some constant $K>0$ depending on $M$ and $\alpha$.
\end{theorem}

\begin{proof} Suppose that (\ref{defcond}) is satisfied  for some constants $K>0$ and $0<\alpha\le 1$. Note that this condition implies that $f$ has a continuous extension on $\overline{\D}$ that satisfies (\ref{defcond}). By (\ref{defcond}) and (\ref{sphecho}) we derive that, for every $z_1,z_2\in \D$,
\begin{equation}\label{assumho}
|f(z_1)-f(z_2)|\le \frac{K}{2}|z_1-z_2|^{\alpha}\sqrt{1+|f(z_1)|^2}\sqrt{1+|f(z_2)|^2}.
\end{equation}

{\bf Case 1:} Assume that $0\in D$ and let $z_0=f^{-1}(0)$. We consider the disk $D(z_0,r_0)$ with $r_0={(1-|z_0|)}/{4}$.

{\bf Case 1(a):} Let $z\in D(z_0,r_0)$. Since $|z-z_0|\le r_0$, we have 
\[ 1-|z|\le |z-z_0|+1-|z_0|\le \frac{5}{4}(1-|z_0|),\]
and thus
\begin{equation}\label{r0}
1-|z_0|\ge \frac{4}{5}(1-|z|).
\end{equation} 
Applying Theorem \ref{koebedistortion} on $D(z_0, 2r_0)$ and Corollary \ref{koebedistance} on $\D$, we have
\begin{align}\label{fder}
f^{\#}(z)=\frac{2|f'(z)|}{1+|f(z)|^2}\le 2|f'(z)|\le 8|f'(z_0)|\le \frac{32\dist (0,\partial D)}{1-|z_0|}.
\end{align}
Let $w\in \partial \D$ such that $1-|z_0|=|w-z_0|$. By the fact that $\dist(0,\partial D)\le |f(w)|$ and by (\ref{assumho}), we deduce that
\[ \frac{\dist(0,\partial D)}{\sqrt{1+\dist(0,\partial D)^2}} \le \frac{|f(w)|}{\sqrt{1+|f(w)|^2}}\le \frac{K}{2}(1-|z_0|)^{\alpha}.\]
This in combination with (\ref{r0}) and (\ref{fder}) gives
\[
f^{\#}(z)\le \left(\frac{5}{4}\right)^{1-\alpha}\frac{16K\sqrt{1+\dist(0,\partial D)^2}}{(1-|z|)^{1-\alpha}}\le \frac{20K(1+\dist(0,\partial D))}{(1-|z|)^{1-\alpha}},
\]
for every $z\in D(z_0,r_0)$.

{\bf Case 1(b):} Let $z\in \D\backslash D(z_0,r_0)$. Then the disk $D(z,r)$ with radius $r=\min \{1-|z|,|z-z_0|\}$ lies in  $\D \backslash \{z_0\}$. We claim that 
\begin{equation}\label{claim1}
\frac{1-|z|}{5}\le r\le 1-|z|.
\end{equation}
Obviously, this holds when $r=1-|z|$. Suppose that  $r=|z-z_0|$. Since $z\in \D\backslash D(z_0,r_0)$ we have $r\ge r_0$. So, 
\[1-|z|\le |z-z_0|+1-|z_0|=r+4r_0\le 5r,\]
which verifies (\ref{claim1}). Now, applying Cauchy's integral formula on the circle $\gamma(t)=z+\frac{r}{2}e^{it}$, for $t\in [0,2\pi)$,  we infer that
\begin{align}
f'(z)&=\frac{1}{2\pi i}\int_{\gamma} \frac{f(w)}{(w-z)^2}dw=\frac{1}{\pi r}\int_{0}^{2\pi }\frac{f(z+\frac{r}{2}e^{it})}{e^{it}}dt \nonumber\\
&=\frac{1}{\pi r}\int_{0}^{2\pi }\frac{f(z+\frac{r}{2}e^{it})-f(z)}{e^{it}}dt \nonumber
\end{align}
and hence
\[|f'(z)|\le \frac{1}{\pi r}\int_{0}^{2\pi }|f(z+\frac{r}{2}e^{it})-f(z)|dt.\]
This in conjunction with (\ref{assumho}) gives
\begin{align}\label{claim2}
f^{\#}(z)&=\frac{2|f'(z)|}{1+|f(z)|^2}\le \frac{2}{\pi r}\int_{0}^{2\pi }\frac{|f(z+\frac{r}{2}e^{it})-f(z)|}{1+|f(z)|^2}dt
\nonumber\\
&\le \frac{K}{2^{\alpha}\pi r^{1-\alpha}}\int_0^{2\pi} \frac{\sqrt{1+|f(z+\frac{r}{2}e^{it})|^2}}{\sqrt{1+|f(z)|^2}}dt 
\end{align}
Applying Theorem \ref{koebedistortion} and Corollary \ref{koebedistance} on $D(z,r)$, we have
\begin{align}\label{claim3}
|f(z+\frac{r}{2}e^{it})|&\le |f(z+\frac{r}{2}e^{it})-f(z)|+|f(z)|\le 2r|f'(z)|+|f(z)| \nonumber\\
&\le 8\dist (f(z),\partial f(D(z,r)))+|f(z)|\le 9|f(z)|.
\end{align}
The last inequality comes from the fact that $0$ lies in the complement of $f(D(z,r))$. Therefore, by (\ref{claim1}), (\ref{claim2}) and (\ref{claim3}) we deduce that
\begin{equation}\label{casefar}
f^{\#}(z)\le \frac{9K10^{1-\alpha}}{(1-|z|)^{1-\alpha}}\le \frac{90K}{(1-|z|)^{1-\alpha}},
\end{equation}
for every $z\in \D\backslash D(z_0,r_0)$.

{\bf Case 2:} Suppose $0\notin D$. Let $z\in \D$ and take $D(z,r)$ with $r=1-|z|$. Working exactly as in Case 1(b) (with the only difference that now $r=1-|z|$) we deduce that (\ref{casefar}) holds for every $z\in \D$.

Finally, combining all the cases above we derive that, for every $z\in \D$,
\[f^{\#}(z)\le \frac{90K(1+\dist(0,\C\backslash D))}{(1-|z|)^{1-\alpha}}.\]

Conversely, suppose there are $M>0$ and $0<\alpha \le 1$ such that, for every $z\in \D$,
\begin{equation}\label{esti}
f^{\#} (z)\le \frac{M}{(1-|z|)^{1-\alpha}}.
\end{equation}
We take $z,w\in \D$. Let $\Gamma$ be the hyperbolic geodesic in $\D$ passing through $z$ and $w$ and having two endpoints $y,y'$ on $\partial \D$ as in Fig.\ \ref{partit}. Also, let $\Gamma_{zw}$ be the subarc of $\Gamma$ joining $z$ to $w$. We set $l:=l(\Gamma)$ and we take the partition $P=\{\gamma_1,\gamma_2,\dots,\gamma'_1,\gamma'_2,\dots\}$ of $\Gamma$ as in Fig.\ \ref{partit} such that
\[l(\gamma_n)=l(\gamma'_n)=\frac{l}{2^{n+1}},\]
for every $n\ge 1$. Next, we estimate $\sigma(f(z),f(w))$. We consider three cases.

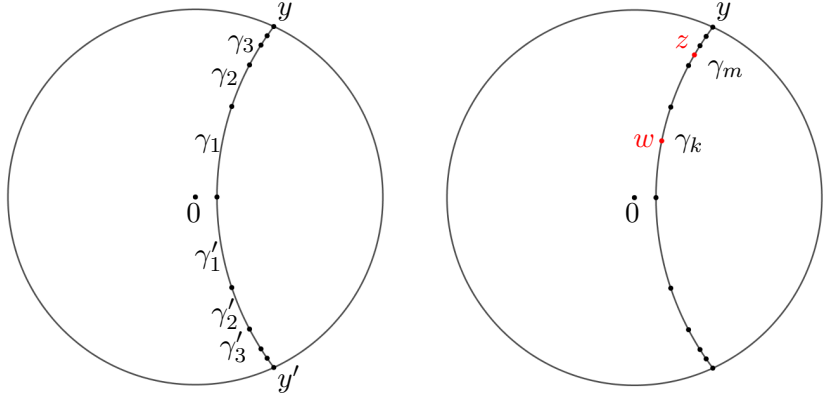
\begin{figure}
	\begin{overpic}[width=11cm]{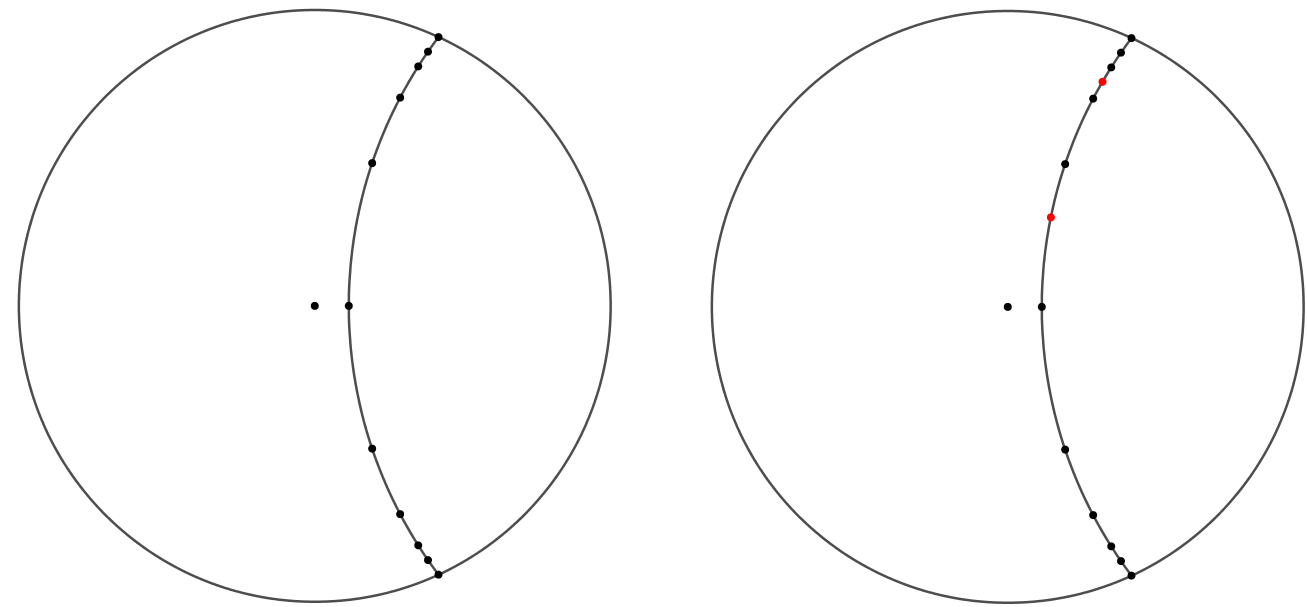}
		\put (23,20) {$0$}	
		\put (76,20) {$0$}
		\put (34,45) {$y$}
		\put (34,0) {$y'$}
		\put (24,29) {$\gamma_1$}	
		\put (26,37) {$\gamma_2$}	
		\put (28,41) {$\gamma_3$}	
		\put (24,15) {$\gamma'_1$}
		\put (26,8) {$\gamma'_2$}	
		\put (27,4) {$\gamma'_3$}
		\put (82,41) {\textcolor{red}{$z$}}
		\put (77,29) {\textcolor{red}{$w$}}
		\put (87,45) {$y$}
		\put (82,29) {$\gamma_k$}	
		\put (86,38) {$\gamma_m$}		
	\end{overpic}
	\caption{Left: The partition of $\Gamma$.\ Right: The points $z,w$.}
	\label{partit}	
\end{figure}

{\bf Case 1:} Let $w\in \gamma_{m-1}$ and $z\in \gamma_{m}$ for some $m> 1$. We observe that
\[|z-w|\le l(\Gamma_{zw})\le l(\gamma_{m-1})+l(\gamma_{m})=\frac{3l}{2^{m+1}}.\]
If $\Gamma_{zy}$ is the subarc of $\Gamma$ joining $z$ to $y$, then this and (\ref{geodesic}) give
\begin{equation}\label{sx2}
1-|z|\ge \frac{2l(\Gamma_{zy})}{\pi}\ge \frac{2l(\gamma_{m})}{\pi}=\frac{l}{\pi2^{m}}\ge\frac{2|z-w|}{3\pi},
\end{equation}
for a constant $c>1$. By (\ref{archord}), (\ref{esti}) and (\ref{sx2}) we have
\begin{align}\label{m1}
\sigma(f(z),f(w))&\le \int_{\Gamma_{zw}} f^{\#}(\zeta)|d\zeta|\le \int_{\Gamma_{zw}} \frac{M|d\zeta|}{(1-|\zeta|)^{1-\alpha}}\le \frac{M l(\Gamma_{zw})}{(1-|z|)^{1-\alpha}}  \nonumber\\
&\le \frac{M\pi|z-w|}{2(1-|z|)^{1-\alpha}} \le M_1 |z-w|^{\alpha}, 
\end{align}
where $M_1=M3^{1-\alpha}(\pi/2)^{2-\alpha}$. We note that the same proof applies when $z,w$ lie in the same arc or in consecutive arcs of the partition $P$.

{\bf Case 2:} Let $w\in \gamma_k$ and $z\in \gamma_{m}$ for some $m,k\ge 1$ such that $k+1<m$. So, $z,w$ do not belong to the same arc or consecutive arcs of $P$ (see Fig.\ \ref{partit}). By (\ref{esti}) we have
\begin{align}
\sigma(f(z),f(w))\le \sum_{i=k}^{m} \int_{\gamma_i} f^{\#}(\zeta)|d\zeta| \le\sum_{i=k}^{m} \int_{\gamma_i}\frac{M}{(1-|\zeta|)^{1-\alpha}}|d\zeta|. \nonumber
\end{align}
If $\zeta\in \gamma_i$, then by (\ref{geodesic}) we deduce that there is $c>1$ such that
\begin{align}\label{relnew}
1-|\zeta|\ge \frac{2l(\Gamma_{\zeta y})}{\pi}\ge\frac{2l(\gamma_i)}{\pi}.
\end{align}
Hence, it follows that
\begin{align}\label{big}
\sigma(f(z),f(w))&\le M\left(\frac{\pi}{2} \right)^{1-\alpha}\sum_{i=k}^{m} \int_{\gamma_i}\frac{|d\zeta|}{l(\gamma_i)^{1-\alpha}}=M\left(\frac{\pi}{2} \right)^{1-\alpha}\sum_{i=k}^{m} l(\gamma_i)^{\alpha}   \nonumber\\
&=M\left(\frac{\pi}{2} \right)^{1-\alpha}l^{\alpha}\sum_{i=k}^{m} \frac{1}{2^{(i+1)\alpha}}\le \frac{M(\pi/2)^{1-\alpha}}{1-1/2^{\alpha}} \frac{l^{\alpha}}{2^{(k+1)\alpha}}.
\end{align}
By (\ref{archord}) we have
\[|z-w|\ge \frac{2}{\pi}l(\Gamma_{zw})\ge \frac{2}{\pi}l(\gamma_{k+1})=\frac{l}{\pi2^{k+1}}\]
and thus (\ref{big}) gives
\begin{equation}\label{m2}
\sigma(f(z),f(w))\le M_2|z-w|^{\alpha},
\end{equation}
where $M_2=M\pi (1/2)^{1-\alpha}/(1-1/2^{\alpha})$. This proof also applies if $w\in \gamma'_k$ and $z\in \gamma'_m$.

{\bf Case 3:}  Let $w\in \gamma'_k$ and $z\in \gamma_{m}$ for some $k,m\ge 1$ such that $\gamma'_k,\gamma_m$ are not consecutive. As in (\ref{relnew}), if $\zeta\in \gamma_i'$ then $1-|\zeta|\ge 2l(\gamma_i')/\pi$. By this, (\ref{esti}) and (\ref{relnew}) we have
\begin{align}
\sigma(f(z),f(w))&\le \sum_{i=1}^{m} \int_{\gamma_i} f^{\#}(\zeta)|d\zeta|+\sum_{i=1}^{k} \int_{\gamma'_i} f^{\#}(\zeta)|d\zeta| \nonumber\\
&\le \sum_{i=1}^{m} \int_{\gamma_i}\frac{M}{(1-|\zeta|)^{1-\alpha}}|d\zeta|+ \sum_{i=1}^{k} \int_{\gamma'_i}\frac{M}{(1-|\zeta|)^{1-\alpha}}|d\zeta| \nonumber\\
&\le M\left(\frac{\pi}{2} \right)^{1-\alpha}\left(\sum_{i=1}^{m} l(\gamma_i)^{\alpha} +\sum_{i=1}^{k} l(\gamma'_i)^{\alpha} \right)
 \nonumber\\
&= M\left(\frac{\pi}{2} \right)^{1-\alpha}l^{\alpha}\left(\sum_{i=1}^{m} \left(\frac{1}{2^{\alpha}} \right)^{i+1} +\sum_{i=1}^{k} \left(\frac{1}{2^{\alpha}} \right)^{i+1} \right) \nonumber\\
&\le \frac{2^{\alpha}M\pi^{1-\alpha}}{1-1/2^{\alpha}} l^{\alpha}. \nonumber
\end{align}
Moreover, by (\ref{archord}) we take
\[|z-w|\ge \frac{2}{\pi}l(\Gamma_{zw})\ge \frac{2}{\pi}l(\gamma_1)=\frac{l}{2\pi}.\]
Thus, we infer that
\begin{equation}\label{m3}
\sigma(f(z),f(w))\le M_3 |z-w|^{\alpha},
\end{equation}
where $M_3=4^{\alpha}M\pi/(1-1/2^{\alpha})$.

Since we have covered all cases, by (\ref{m1}), (\ref{m2}) and (\ref{m3}) we deduce that there is a positive constant $K=\max\{M_1,M_2,M_3\}$ depending on $M$ and $\alpha$ such that
\[\sigma(f(z),f(w))\le K|z-w|^{\alpha},\]
for every $z,w\in \D$. 
\end{proof}

\section{Unbounded H\"{o}lder domains and hyperbolic distance}\label{defhyperbolic}

Let $D \subsetneq \C$ be an unbounded simply connected domain and let $w_0\in D$. If  $\tilde{w}$ is a prime end of $D$, then $\Gamma$ denotes the hyperbolic geodesic in $D$ joining $w_0$ to $\tilde{w}$ and $\Gamma(w,\tilde{w})$ is the subarc of $\Gamma$ joining $w$ to $\tilde{w}$. Let $0<\alpha\le 1$. We consider the following conditions: 

\begin{enumerate}[label=(C\arabic*)]
\item There is a constant $C$ such that, for every $w\in \Gamma$,
\[h_D(w_0,w)\le C+\frac{1}{\alpha}\log\frac{1}{l_{\sigma}(\Gamma(w,\tilde{w}))}.\]
\item  There is a constant $C$ such that, for every $w\in \Gamma$, 
\[ h_D(w_0,w)\le C+\frac{1}{\alpha}\log\frac{1}{\dist_{\sigma}(w,\tilde{w})},\]
where $\dist_{\sigma}(w,\tilde w)$ denotes the spherical distance from $w$ to the impression of the prime end $\tilde w$. 
\item  There is a constant $C$ such that, for every $w\in \Gamma$,
\[h_D(w_0,w)\le C+\frac{1}{\alpha}\log\frac{1}{\dist_\sigma(w,\partial D)}.\]
\end{enumerate}
Note that condition (C1) implies condition (C2) and this implies condition (C3) because
\[l_{\sigma}(\Gamma (w,\tilde{w}))\ge \dist_\sigma(w,\tilde{w})\ge \sigma(w,\partial D).\]

\begin{theorem}\label{hyho}
Let $D\subsetneq \C$ be an unbounded simply connected domain and $f$ be a Riemann mapping from $\D$ onto $D$. If $f(0)=w_0$, then the following statements are equivalent.
\begin{enumerate}[label=\normalfont(\roman*)]
\item $D$ is an unbounded $\alpha$-H\"{o}lder domain.
\item For each prime end $\tilde{w}$ of $D$, the hyperbolic geodesic $\Gamma$ in $D$ joining $w_0$ to $\tilde{w}$ satisfies one of conditions {\rm (C1)}, {\rm (C2)} or {\rm (C3)} for some constant $C$ independent of $\tilde{w}$.
\item There is a constant $C$ such that, for every $w\in D$,
\[h_D(w_0,w)\le C+\frac{1}{\alpha}\log\frac{1}{\dist_\sigma(w,\partial D)}.\]
\end{enumerate}
The constants depend quantitatively on each other and on $\dist (0,\C \backslash D)$.
\end{theorem}

\begin{proof} First, we prove that (i) implies (ii). Suppose that $D$ is an unbounded $\alpha$-H\"{o}lder domain. It suffices to prove that condition (C1) is satisfied. Let $\tilde{w}$ be a prime end of $D$ and set $e^{i\theta}=f^{-1}(\tilde{w})$ for some $\theta \in [0,2\pi)$. By Theorem \ref{hoderivative} it follows that there is a constant $M>0$ such that, for every $z\in \D$,
\[f^{\#} (z)\le \frac{M}{(1-|z|)^{1-\alpha}}.\]
Let $0<r<1$. If $\gamma(t)=te^{i\theta}$ for $r\le t<1$, then we have
\[
l_{\sigma}(\Gamma (f(re^{i\theta}),f(e^{i\theta})))=\int_{\gamma} f^{\#}(\zeta)|d\zeta| \le M \int_r^1 \frac{dt}{(1-t)^{1-\alpha}} =\frac{M}{\alpha}(1-r)^{\alpha} \]
and thus
\[\frac{1}{\alpha}\log \frac{1}{l_{\sigma}(\Gamma (f(re^{i\theta}),f(e^{i\theta})))}\ge \frac{1}{\alpha}\log\frac{\alpha}{M}+\log \frac{1}{1-r}.\]
So, for every $0<r<1$, we deduce that
\begin{align}
h_{D}(f(0),f(re^{i\theta}))&=h_{\D} (0,re^{i\theta})=\log \frac{1+r}{1-r}\le \log 2 + \log\frac{1}{1-r} \nonumber\\
&\le C+\frac{1}{\alpha}\log \frac{1}{l_{\sigma}(\Gamma (f(re^{i\theta}),f(e^{i\theta})))}, \nonumber
\end{align}
where $C=\log 2+ (1/\alpha)\log (M/\alpha)$. In other words, for every $w\in \Gamma$,
\[h_{D}(w_0,w)\le C+\frac{1}{\alpha}\log \frac{1}{l_{\sigma}(\Gamma (w,\tilde{w}))}.\]

To prove that (ii) implies (iii), it suffices to show that if condition (C3) is satisfied then (iii) follows. Obviously, every $w\in D$ lies on a hyperbolic geodesic in $D$ joining $w_0$ to a prime end $\tilde{w}$ of $D$ and thus, by assumption,
\[h_D(w_0,w)\le C+\frac{1}{\alpha}\log\frac{1}{\dist_\sigma(w,\partial D)}.\]

Finally, we prove that (iii) implies (i). We suppose that there are constants $C\in \mathbb{R}$ and $\alpha\in (0,1]$ such that, for every $w\in D$,
\begin{equation}\label{rel1}
h(w_0,w)\le C +\frac{1}{\alpha}\log \frac{1}{\dist_{\sigma}(w, \partial D)}.
\end{equation}
If $z\in \D$, then
\[h_D(f(0),f(z))=\log\frac{1+|z|}{1-|z|}\ge \log\frac{1}{1-|z|}.\]
This in combination with (\ref{rel1}) implies that, for $z\in \D$,
\begin{align}\label{dsigma}
\dist_{\sigma}(f(z), \partial D)\le e^{\alpha C}(1-|z|) ^{\alpha}.
\end{align}

{\bf Case 1:} We assume that $0\in D$ and we set $z_0=f^{-1}(0)$. Consider the disk $D(z_0,r_0)$, where $r_0={(1-|z_0|)}/{4}$.

{\bf Case 1(a):} Let $z\in D(z_0,r_0)$. By (\ref{sphecho}) and (\ref{dsigma}) we deduce that
\begin{align}
e^{\alpha C}(1-|z_0|) ^{\alpha} &\ge \dist_{\sigma}(0, \partial D)\ge \chi (0, \partial D)=\min \left\{\frac{2|w|}{\sqrt{1+|w|^2}}:w\in \partial D \right\} \nonumber\\
&\ge \frac{2\dist (0,\partial D)}{\sqrt{1+\dist(0,\partial D)^2}} \nonumber
\end{align}
and thus
\[2\dist (0,\partial D)\le e^{\alpha C}(1-|z_0|) ^{\alpha}\sqrt{1+\dist(0,\partial D)^2}. \]
This in conjunction with (\ref{fder}) implies that
\[f^{\#}(z)\le \frac{ 16e^{\alpha C}\sqrt{1+\dist(0,\partial D)^2}}{(1-|z_0|)^{1-\alpha}}.\]
By this and (\ref{r0}) we infer that, for every $z\in D(z_0,r_0)$,
\begin{equation}\label{result1}
f^{\#}(z)\le \frac{20e^{\alpha C}\sqrt{1+\dist(0,\partial D)^2}}{(1-|z|)^{1-\alpha}}.
\end{equation}

{\bf Case 1(b):} Let $z\notin D(z_0,r_0)$. We claim that
\begin{equation}\label{dist0}
{\dist(0,\partial D)}\le 32|f(z)|.
\end{equation} 
Indeed, applying Corollary \ref{koebedistance} on $D(z_0,r_0)$, we have
\begin{align}
\dist (0,\partial f(D(z_0,r_0)))\ge \frac{r_0}{4} |f'(z_0)|=\frac{1}{16}(1-|z_0|)|f'(z_0)| \nonumber.
\end{align}
Also, applying Corollary \ref{koebedistance} on $\D$ we take
\[2|f'(z_0)|(1-|z_0|)\ge \dist (0,\partial D).\]
Combining these results, we derive that 
\[\dist (0,\partial f(D(z_0,r_0)))\ge\frac{\dist (0,\partial D)}{32}.\]
This implies that $D(0,\dist(0,\partial D)/32)\subset f(D(z_0,r_0))$. Given that $f(z)\notin f(D(z_0,r_0))$, this proves (\ref{dist0}). By (\ref{dist0}) and Corollary \ref{koebedistance} it follows that
\begin{align}\label{fz}
f^{\#}(z)&\le \frac{8\dist (f(z),\partial D)}{(1-|z|)(1+|f(z)|^2)}\le 8\frac{\dist(0,\partial D)+|f(z)|}{(1-|z|)(1+|f(z)|^2)} \nonumber\\
&\le  \frac{264}{1-|z|}\frac{|f(z)|}{1+|f(z)|^2}.
\end{align}
Now, we are going to use our assumption (\ref{dsigma}). Let $w\in \partial D$ be a point for which $\chi (f(z),\partial D)=\chi (f(z),w)$. If $|w|\ge 2|f(z)|$ then
\begin{align}
\chi (f(z),w)&\ge \frac{2(|w|-|f(z)|)}{\sqrt{1+|f(z)|^2}\sqrt{1+|w|^2}}\ge \frac{2|f(z)|}{\sqrt{1+|f(z)|^2}\sqrt{1+4|f(z)|^2}}, \nonumber
\end{align}
where we used the fact that the function $g(x)=(x-c)/\sqrt{1+x^2}$ for $c>0$ and $x\ge 2c$ is increasing. This in conjunction with (\ref{sphecho}) and (\ref{dsigma}) gives
\begin{equation}\label{sxx11}
e^{\alpha C}(1-|z|)^{\alpha}\ge \chi (f(z),w)\ge \frac{|f(z)|}{1+|f(z)|^2}.
\end{equation}
Therefore, combining this  with (\ref{fz}) we derive that
\begin{equation}
f^{\#}(z)\le \frac{264e^{\alpha C}}{(1-|z|)^{1-\alpha}}. \nonumber
\end{equation}
If $|w|\le 2|f(z)|$ then (\ref{sphecho}) and (\ref{dsigma}) give
\begin{equation}\label{sxx22}
e^{\alpha C}(1-|z|)^{\alpha}\ge \chi (f(z),w)=\frac{2|f(z)-w|}{\sqrt{1+|f(z)|^2}\sqrt{1+|w|^2}}\ge \frac{\dist (f(z),\partial D)}{1+|f(z)|^2}.
\end{equation}
By this  and Corollary \ref{koebedistance} we have
\begin{align}
f^{\#}(z)\le \frac{8\dist (f(z),\partial D)}{(1-|z|)(1+|f(z)|^2)}\le \frac{8e^{\alpha C}}{(1-|z|)^{1-\alpha}}.\nonumber
\end{align}
Thus, in any case, for $z\notin D(z_0,r_0)$,
\begin{equation}\label{result22}
f^{\#}(z)\le \frac{264e^{\alpha C}}{(1-|z|)^{1-\alpha}}.
\end{equation}

{\bf Case 2:} Suppose that $0\notin D$. Let $z\in \D$. By Corollary \ref{koebedistance} we have
\[
f^{\#}(z)\le \frac{8\dist (f(z),\partial D)}{(1-|z|)(1+|f(z)|^2)}\le \frac{8|f(z)|}{(1-|z|)(1+|f(z)|^2)}.\]
Let $w$ be a point on $\partial D$ for which $\chi (f(z),\partial D)=\chi (f(z),w)$. Now, working as in Case 1(b), we infer that (\ref{result22}) holds for every $z\in \D$.

Finally, combining all the cases above we derive that, for every $z\in \D$,
\[f^{\#}(z)\le\frac{264e^{\alpha C}\sqrt{1+\dist (0, \C\backslash D)^2}}{(1-|z|)^{1-\alpha}}.\]
So, Theorem \ref{hoderivative} implies that $D$ is an unbounded $\alpha$-H\"{o}lder domain.
\end{proof}

\section{Hardy number of unbounded H\"{o}lder domains}\label{application}

\begin{theorem} Let $D\subsetneq \C$ be an unbounded simply connected domain and $0<\alpha \le 1$. If $D$ is an unbounded $\alpha$-H\"{o}lder domain, then
\begin{equation} \label{hardyholder}
{\rm h}(D)\le 1/\alpha
\end{equation}
and this estimate is sharp.
\end{theorem}

\begin{proof} Let $f$ be a Riemann mapping from $\D$ onto $D$ and set $f(0)=w_0$. Since $D$ is an unbounded $\alpha$-H\"{o}lder domain, $f$ has a continuous and surjective extension from $\overline{\D}$ onto $\overline{D}\cup \{\infty\}$. Let $\zeta$ be a preimage of $\infty$ under $f^{-1}$ and take the hyperbolic geodesic $\gamma$ in $\D$ joining $0$ to $\zeta$. Then $\Gamma:=f(\gamma)$ is a hyperbolic geodesic joining $w_0$ to $\infty$ in $D$.  If $C_r=\{ z:|z|=r \}$ for $r>0$, then ${\rm h}(D)$ can be found with the aid of the hyperbolic distance ${h_D (w_0,C_r)}$ as follows 
\begin{equation} \label{hardynumber}
{\rm h}(D)=\liminf_{r\to \infty}\frac{h_D (w_0,C_r)}{\log r}
\end{equation}
(see \cite[Theorem 1.1]{Kar}). Let $w_r \in \Gamma \cap C_r$ for $r>|w_0|$. Theorem \ref{hyho} implies that there is a constant $C$ such that
\[h_D (w_0,C_r)\le h_D(w_0,w_r)\le C+\frac{1}{\alpha}\log\frac{1}{\sigma(w_r,\infty)},\]
for every $r>|w_0|$. Also, by (\ref{sphecho}) we have
\[\sigma(w_r,\infty)\ge \chi (w_r,\infty) = \frac{2}{\sqrt{1+|w_r|^2}}=\frac{2}{\sqrt{1+r^2}}.\]
Therefore, for every $r>|w_0|$,
\[h_D (w_0,C_r)\le C+\frac{1}{\alpha}\log\frac{\sqrt{1+r^2}}{2}.\]
This in combination with (\ref{hardynumber}) implies that ${\rm h}(D)\le 1/\alpha$.

To prove that this estimate is sharp we consider a sector domain with opening $\theta\in(0, 2\pi)$,
\[S_{\theta}=\{z:|\Arg z|<\theta/2\}.\]
Then the conformal mapping
\[g(z)=\left(\frac{1+z}{1-z}\right)^{\theta/\pi}\]
maps $\D$ onto $S_{\theta}$. Next, we prove that $S_{\theta}$ is an unbounded $\alpha$-H\"{o}lder domain and we find $\alpha$. For $z\in \D$, we have
\begin{align}\label{gderivative}
g^{\#}(z)=\frac{2|g'(z)|}{1+|g(z)|^2}=\frac{4\theta}{\pi}\frac{|1+z|^{\theta/\pi-1}|1-z|^{\theta/\pi -1}}{|1-z|^{2\theta/\pi}+|1+z|^{2\theta/\pi}}.
\end{align}
Recall that if $a,b,k>0$ then
\[\frac{(a+b)^k}{2^k} \le a^k+b^k.\]
Therefore, for $z\in \D$, we have
\[|1-z|^{2\theta/\pi}+|1+z|^{2\theta/\pi} \ge 2^{-\theta/\pi}(|1-z|^2+|1+z|^2)^{\theta/\pi} = (1+|z|^2)^{\theta/\pi}\ge 1.\]
This and (\ref{gderivative}) imply that, for $z\in \D$,
\begin{equation}\label{thetapi}
g^{\#}(z)\le \frac{4\theta}{\pi}|1+z|^{\theta/\pi-1}|1-z|^{\theta/\pi -1}.
\end{equation}
If $\theta/\pi \ge 1$ then by (\ref{thetapi}) we infer that, for every $z\in \D$,
\[ g^{\#}(z)\le \frac{4^{\theta/\pi}\theta}{\pi}.\]
By Theorem \ref{hoderivative} we deduce that $S_{\theta}$ is an unbounded 1-H\"{o}lder domain. If $0<\theta/\pi < 1$ then, for every $z\in \D$,
\[|1+z|^{1-\theta/\pi}|1-z|^{1-\theta/\pi}=|1-z^2|^{1-\theta/\pi}\ge (1-|z|^2)^{1-\theta/\pi}\ge (1-|z|)^{1-\theta/\pi}\]
and thus by (\ref{thetapi}) we have
\[ g^{\#}(z)\le \frac{4\theta}{\pi}\frac{1}{(1-|z|)^{1-\theta/\pi}}.\]
By Theorem \ref{hoderivative} we derive that $S_{\theta}$ is an unbounded $\theta/\pi$-H\"{o}lder domain.

In conclusion, we have shown that if $0<\theta <\pi$, then $S_{\theta}$ is an unbounded $\theta/\pi$-H\"{o}lder domain and if $\theta\ge \pi$, then $S_{\theta}$ is an unbounded 1-H\"{o}lder domain. Combining this with the fact that ${\rm h}(S_{\theta})=\pi/\theta$ (see \cite{Han}) we conclude that (\ref{hardyholder}) is sharp. 
\end{proof} 

\section{Unbounded H\"{o}lder domains and classical H\"{o}lder domains}\label{defholderbounded}

In this section we show the connection of the definition of unbounded H\"{o}lder domains to the definition of classical (bounded) H\"{o}lder domains. 

Let $D\subsetneq \C$ be an unbounded simply connected domain and $w_0\in D$. We set $r=\dist (w_0,\partial D)/2$, and we consider the function  
\[g(w)=\frac{r}{w-w_0},\]
for $w\in D$. We set $D'=g(D)$ and $D_0=D'\cap D(0,4)$. Note that $D'$ is a simply connected domain in $\C_{\infty}$ that contains infinity and its boundary is bounded, and $D_0$ is a bounded doubly connected domain (see Fig.\ \ref{generalfig}). Let $0<\alpha \le 1$. In this section we prove that \textit{``$D$ is an unbounded H\"{o}lder domain if and only if $D_0$ is a (bounded) H\"{o}lder domain".}

Before showing it, we establish two auxiliary lemmas.

\begin{figure}
	\begin{overpic}[width=\linewidth]{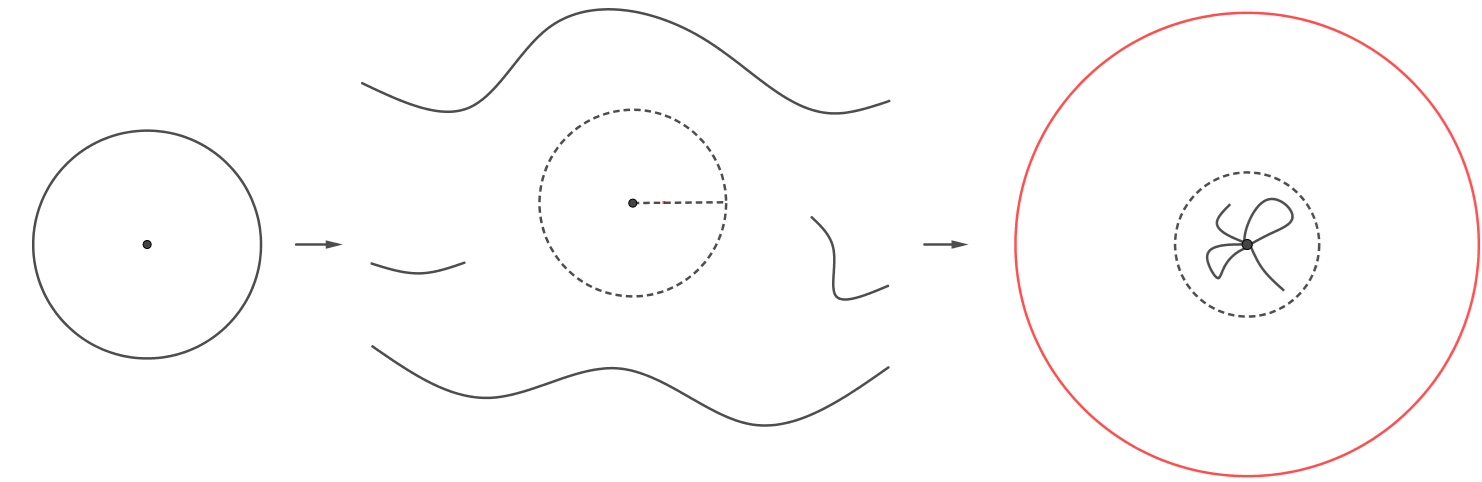}
		\put (31,7) {$D$}
		\put (40,20) {$w_0$}
		\put (45,17) {$r$}
		\put (63,18) {$g$}
		\put (83.5,13) {$0$}
		\put (20,18) {$f$}
		\put (9,17) {$0$}
		\put (95,1) {$D'$}
		\put (80,2) {$D_0$}
	\end{overpic}
	\caption{The domains $D,D'$ and $D_0$.}
	\label{generalfig}	
\end{figure}
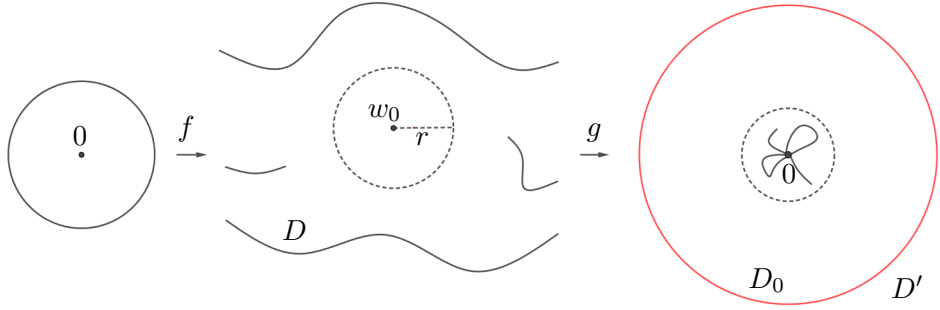

\begin{lemma}\label{coro} With the notation defined above, there are constants $c_1,c_2>0$ depending on $|w_0|$ and $r$ such that, for every $w\in D$,
\[c_1\dist_{\sigma} (w,\partial D) \le \dist_{\sigma}(g(w),\partial D')\le c_2\dist_{\sigma} (w,\partial D).\]
\end{lemma}

\begin{proof} By (\ref{sphecho}) we derive that, for $z,w\in D$,
\begin{align}\label{ratio}
\frac{\sigma(z,w)}{\sigma(g(z),g(w))}&\ge \frac{2}{\pi} \frac{|z-w|}{\sqrt{1+|z|^2}\sqrt{1+|w|^2}}\frac{\sqrt{1+|g(z)|^2}\sqrt{1+|g(w)|^2}}{|g(z)-g(w)|} \nonumber\\
&=\frac{2}{\pi} \frac{\sqrt{r^2+|z-w_0|^2}\sqrt{r^2+|w-w_0|^2}}{r\sqrt{1+|z|^2}\sqrt{1+|w|^2}}
\end{align}
and, similarly,
\begin{align}\label{ratio2}
\frac{\sigma(z,w)}{\sigma(g(z),g(w))}\le \frac{\pi}{2} \frac{\sqrt{r^2+|z-w_0|^2}\sqrt{r^2+|w-w_0|^2}}{r\sqrt{1+|z|^2}\sqrt{1+|w|^2}}.
\end{align}
By (\ref{ratio}) and (\ref{ratio2}) it suffices to prove that, for every $w\in D$, the fraction
\begin{equation}\label{bounds}
\frac{r^2+|w-w_0|^2}{r(1+|w|^2)}
\end{equation}
is bounded from above and below by positive constants that depend only on $r$ and $|w_0|$. First, suppose $w_0=0$. Then, for every $w\in D$,
\begin{equation}\label{uplow1}
\frac{\min \{r,1\}^2}{r}\le \frac{r^2+|w|^2}{r(1+|w|^2)}\le \frac{\max \{r,1\}^2}{r}.
\end{equation}
Now, suppose that $|w_0|> 0$. If $\frac{r}{r+1}|w|>|w_0|$, then
\[\frac{|w|}{r+1}<|w|-|w_0|\le |w-w_0|\le |w|+|w_0|\le \frac{2r+1}{r+1} |w|.\]
Hence, we deduce that
\begin{equation}\label{uplow2}
\frac{1}{r}\min \left\{r, \frac{1}{r+1} \right\}^2\le \frac{r^2+|w-w_0|^2}{r(1+|w|^2)}\le \frac{1}{r}\max \left\{r,\frac{2r+1}{r+1} \right\}^2.
\end{equation}
If $\frac{r}{r+1}|w|\le|w_0|$, then
\begin{equation}\label{uplow3}
\frac{r}{1+\left( \frac{r+1}{r}\right)^2|w_0|^2} \le \frac{r^2+|w-w_0|^2}{r(1+|w|^2)}\le \frac{r^2+\left(\frac{2r+1}{r} \right)^2|w_0|^2}{r}.
\end{equation}
Taking the maximum of the upper bounds in (\ref{uplow1}), (\ref{uplow2}) and (\ref{uplow3}) and the minimum of the lower bounds we infer that, for every $w\in D$, the fraction (\ref{bounds}) is bounded by positive constants depending on $r$ and $|w_0|$.
\end{proof}

\begin{lemma}\label{lemmahyqua} With the notation defined before, for every $y \in D_0$,
\[\lambda_{D'}(y)\delta_{D'}(y)\ge 1/300,\]
where $\lambda_{D'}(y)$ is the hyperbolic density and $\delta_{D'}(y)=\dist (y,\partial D')$.
\end{lemma}

\begin{proof} Let $f$ be a Riemann mapping from $\D$ onto $D$ with $f(0)=w_0$. Also, let $h=g\circ f$ and set $\Delta=\{z:0.01<|z|<1\}$. First, we prove that 
\[\lambda_{D'}(y)\delta_{D'}(y)\ge 1/300,\]
for every $y \in h(\Delta)$. For $1/2<|z|<1$, we apply Corollary \ref{koebedistance} on the disk $D(z,1-|z|)\subset \D\backslash \{0\}$ and we infer that
\begin{equation}\label{boundary}
|h'(z)|(1-|z|)\le 4\dist (h(z),\partial D').
\end{equation}
For $1/100<|z|\le 1/2$, we apply  Corollary \ref{koebedistance} on the disk $D(z,|z|)\subset \D\backslash \{0\}$ and we have
\[|h'(z)||z|\le 4\dist (h(z),\partial D').\]
Since $|z|>1/100$ we have
\begin{equation}\label{zero}
|h'(z)|\le 400\dist (h(z),\partial D').
\end{equation}
For every $z\in \D\backslash \{0\}$, we have
\begin{equation}\label{metric}
\lambda_{D'}(h(z))|h'(z)|=\lambda_{\D}(z)=\frac{2}{1-|z|^2}
\end{equation}
(see \cite{Bea}). Therefore, if $1/2<|z|<1$, by (\ref{boundary}) and (\ref{metric}) we deduce that
\begin{align}
\lambda_{D'}(h(z))\ge \frac{1}{(1-|z|)|h'(z)|}\ge \frac{1}{4\dist (h(z),\partial D')}=\frac{1}{4\delta_{D'}(h(z))} \nonumber
\end{align}
and if $1/100<|z|\le 1/2$, then (\ref{zero}) and (\ref{metric}) imply that
\begin{align}
\lambda_{D'}(h(z))\ge \frac{2}{(1+|z|)|h'(z)|}\ge \frac{1}{300\dist (h(z),\partial D')}=\frac{1}{300\delta_{D'}(h(z))}. \nonumber
\end{align} 
Combining the relations above, we conclude that, for every $z\in \Delta$,
\[\lambda_{D'}(h(z))\delta_{D'}(h(z))\ge 1/300\]
and thus, for every $y\in h(\Delta)$,
\begin{equation}\label{firststep}
\lambda_{D'}(y)\delta_{D'}(y)\ge 1/300.
\end{equation}

Now, we show that  $D_0 \subset h(\Delta)$. Recall that $2r=\dist (w_0,\partial D)$. Since $f$ is conformal on $\D$, by Corollary \ref{koebedistance} we infer that
\[{|f'(0)|}\le 8r.\]
Moreover, by Theorem \ref{koebedistortion}, for $|z|\le 0.01$, we have
\[|f(z)-w_0|\le |f'(0)|\frac{100}{99^2}.\]
Combining the relations above we deduce that, for $|z|\le 0.01$,
\[|f(z)-w_0|\le r\frac{800}{99^2}<\frac{r}{4}.\]
Thus, $f(\overline{D}(0,0.01))\subset D(w_0,r/4)$. This implies that $D\backslash \overline{D}(w_0,r/4) \subset f(\Delta)$ and hence $D_0 \subset h(\Delta)$. This and (\ref{firststep}) gives the desired result.
\end{proof}

\begin{theorem}\label{lastresult} With the notation defined before, $D$ is an unbounded H\"{o}lder domain if and only if $D_0$ is a (bounded) H\"{o}lder domain. 
\end{theorem}

\begin{proof} Suppose that $D$ is an unbounded H\"{o}lder domain. We fix $w^* \in D\backslash \overline{D}(w_0,r)$. Theorem \ref{hyho} implies that there are $\alpha\in (0,1]$ and $C\in \mathbb{R}$ such that, for every $w\in D$,
\[h_D(w^*,w)\le C+\frac{1}{\alpha}\log\frac{1}{\dist_\sigma(w,\partial D)}.\]
Let $w\in D\backslash \overline{D}(w_0,r/4)$ and set $g(w)=z$ and $g(w^*)=z^*$. Then $z^*\in D_0 \cap \D$ and $z\in D_0$ (see Fig.\ \ref{cases12}). By Lemma \ref{coro}, (\ref{sphecho}) and the fact that $\partial D',\{ z \} \subset D(0,4)$, we infer that
\[ \dist_\sigma(w,\partial D)\ge \frac{1}{c_2} \dist_\sigma(z,\partial D') \ge \frac{2}{17c_2}\dist (z,\partial D')\ge \frac{2}{17c_2}\dist (z,\partial D_0),\]
for some constant $c_2$ depending on $r$ and $|w_0|$. Moreover, $h_D(w^*,w)=h_{D'}(z^*,z)$. So, combining these relations we deduce that, for every $z\in D_0$,
\begin{equation}\label{general}
h_{D'}(z^*,z)\le C'+\frac{1}{\alpha}\log\frac{1}{\dist (z,\partial D_0)},
\end{equation}
where $C'= C+(1/\alpha)\log (17c_2/2)$. Now, let $\Gamma$ be the hyperbolic geodesic joining $z^*$ to $z$ in $D'$. If $\Gamma \subset D_0$ then by Lemma \ref{lemmahyqua} we have
\begin{equation}\label{cas1}
k_{D'}(z^*,z)\le 300h_{D'} (z^*,z).
\end{equation}
If $\Gamma \not\subset D_0$, then we consider the following cases.

{\bf Case 1:} If $z\in D_0 \cap \overline{D}(0,2)$ then $\Gamma=\gamma_1 \cup \gamma_2 \cup \gamma_3$, where $\gamma_1$ is the connected subarc of $\Gamma$ joining $z^*$ to $C(0,2)$ and  $\gamma_2$ is the connected subarc of $\Gamma$ joining $z$ to $C(0,2)$ and $\gamma_3 =\Gamma \backslash (\gamma_1 \cup \gamma_2)$. See Fig.\ \ref{cases12} (left). Let $\gamma$ be the shortest arc on $C(0,2)$ joining the endpoints of $\gamma_1$ and $\gamma_2$ (see  Fig.\ \ref{cases12}). Since, for every $\zeta\in \gamma$, $\dist (\zeta, \partial D_0)\ge 1$,  then

\[\int_{\gamma}\delta^{-1}_{D_0}(\zeta)|d\zeta|=\int_{\gamma} \frac{|d\zeta|}{\dist (\zeta, \partial D_0)}\le l(\gamma)\le 4\pi.\]
Also, if $\zeta\in \gamma_1$ or $\gamma_2$ then $\dist (\zeta,\partial D_0)=\dist (\zeta,\partial D')$. So, by these results and Lemma \ref{lemmahyqua} we have
\begin{align}\label{less2} 
k_{D_0} (z^*,z)&\le \int_{\gamma_1}\delta^{-1}_{D_0}(\zeta)|d\zeta|+\int_{\gamma_2}\delta^{-1}_{D_0}(\zeta)|d\zeta|+\int_{\gamma}\delta^{-1}_{D_0}(\zeta)|d\zeta| \nonumber\\
&\le \int_{\gamma_1}\delta^{-1}_{D'}(\zeta)|d\zeta|+\int_{\gamma_2}\delta^{-1}_{D'}(\zeta)|d\zeta|+4\pi \nonumber\\
&\le 4\pi+300\left(\int_{\gamma_1}\lambda_{D'}(\zeta)|d\zeta|+\int_{\gamma_2}\lambda_{D'}(\zeta)|d\zeta|\right) \nonumber\\
&\le 4\pi+300 h_{D'} (z^*,z).
\end{align}

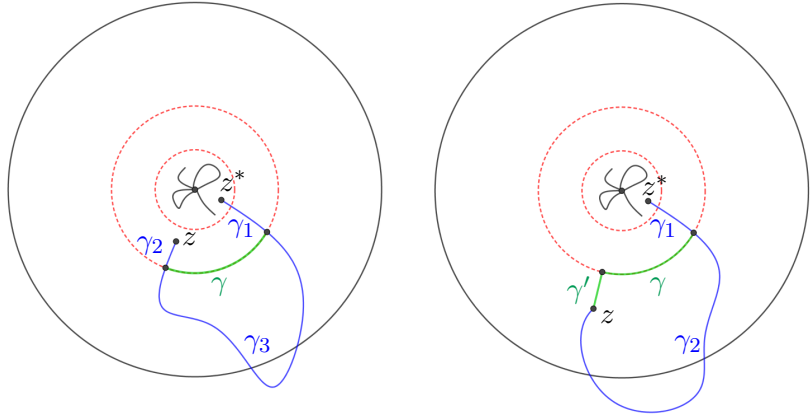
\begin{figure}
	\begin{overpic}[width=11cm]{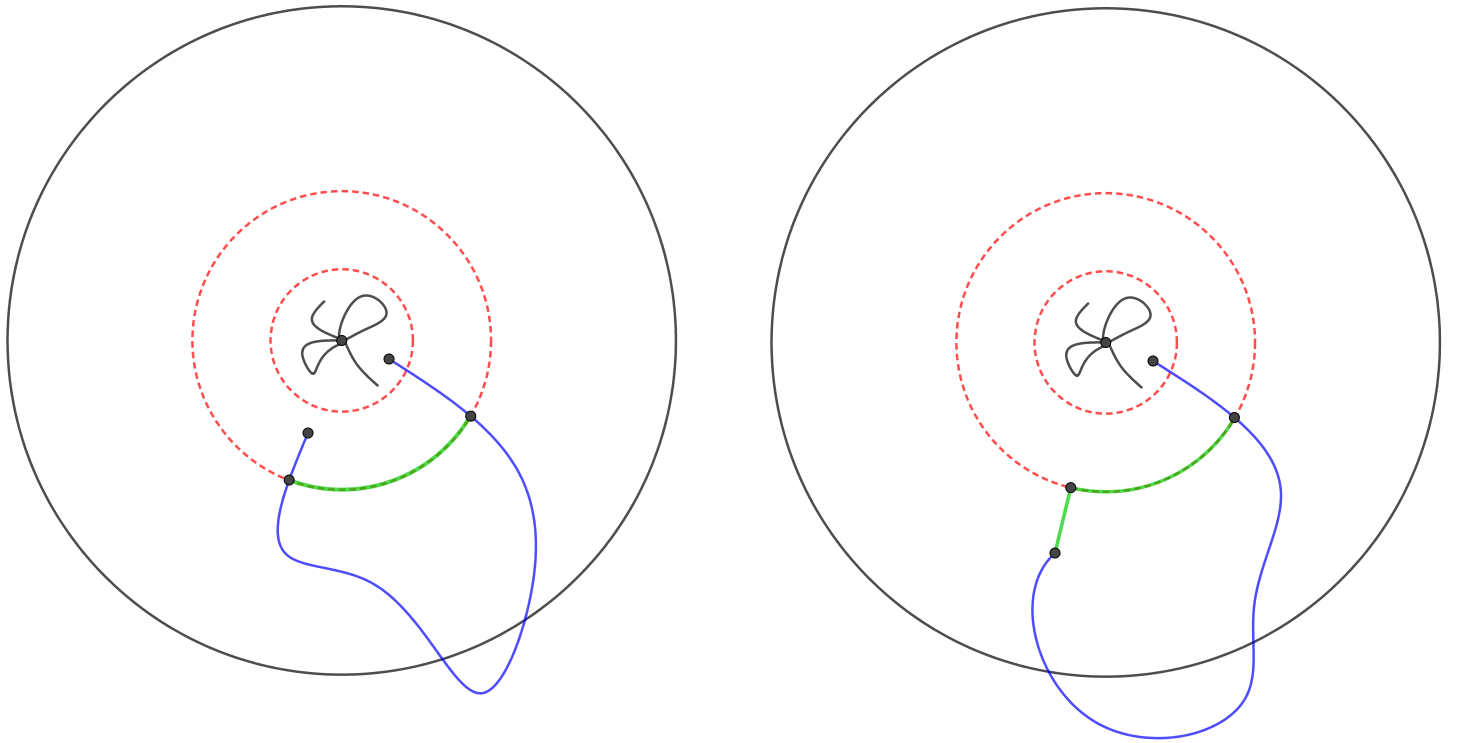}
		\put (26,27) {$z^*$}
		\put (77,26.5) {$z^*$}
		\put (21.5,20.5) {$z$}
		\put (72,11) {$z$}
		\put (16,20) {\textcolor{blue}{$\gamma_2$}}
		\put (27,22) {\textcolor{blue}{$\gamma_1$}}
		\put (25,15) {\textcolor{ForestGreen}{$\gamma$}}
		\put (29,8) {\textcolor{blue}{$\gamma_3$}}
		\put (78,22) {\textcolor{blue}{$\gamma_1$}}
		\put (81,8) {\textcolor{blue}{$\gamma_2$}}
		\put (78,15) {\textcolor{ForestGreen}{$\gamma$}}
		\put (68,14) {\textcolor{ForestGreen}{$\gamma'$}}
	\end{overpic}
	\caption{Left: Case 1. Right: Case 2.}
	\label{cases12}	
\end{figure}

{\bf Case 2:} If $z\in D_0\backslash \overline{D}(0,2)$ (and thus $2< |z|<4$), then $\Gamma=\gamma_1 \cup \gamma_2 $ where $\gamma_1$ is the connected subarc of $\Gamma$ joining $z^*$ to $C(0,2)$ and $\gamma_2 =\Gamma \backslash \gamma_1$. See Fig.\ \ref{cases12} (right). Let $\gamma'$ be the radial segment joining $z$ to $C(0,2)$ and $\gamma$ be the shortest arc on $C(0,2)$ joining the endpoints of $\gamma_1$ and $\gamma'$ (see Fig.\ \ref{cases12}). If $3<|z|<4$, then
\begin{align}
\int_{\gamma}\delta^{-1}_{D_0}(\zeta)|d\zeta|+\int_{\gamma'}\delta^{-1}_{D_0}(\zeta)|d\zeta|&=\int_{\gamma} \frac{|d\zeta|}{\dist (\zeta, \partial D_0)}+\int_{\gamma'} \frac{|d\zeta|}{\dist (\zeta, \partial D_0)} \nonumber\\
&\le 4\pi +\int_2^3 dx +\int_3^{|z|}\frac{dx}{4-x}\nonumber\\
&=4\pi+1+\log\frac{1}{4-|z|} \nonumber\\
&=1+4\pi +\log\frac{1}{\dist (z,\partial D_0)}.
\end{align}
Hence, this and Lemma \ref{lemmahyqua} give
\begin{align}\label{more2}
k_{D_0} (z^*,z)&\le \int_{\gamma}\delta^{-1}_{D_0}(\zeta)|d\zeta|+\int_{\gamma'}\delta^{-1}_{D_0}(\zeta)|d\zeta|+\int_{\gamma_1}\delta^{-1}_{D_0}(\zeta)|d\zeta| \nonumber\\
&\le 1+4\pi +\log\frac{1}{\dist (z,\partial D_0)}+\int_{\gamma_1}\delta_{D'}^{-1}(\zeta)|d\zeta| \nonumber\\
&\le 1+4\pi +\log\frac{1}{\dist (z,\partial D_0)}+300\int_{\gamma_1}\lambda_{D'}(\zeta)|d\zeta| \nonumber\\
&\le 1+4\pi +\log\frac{1}{\dist (z,\partial D_0)}+300h_{D'} (z^*,z).
\end{align}
If $2<|z|\le 3$, then working as before we have
\begin{equation}\label{bet23}
 k_{D_0} (z^*,z)\le 1+4\pi +300h_{D'} (z^*,z).
\end{equation}

Concluding, (\ref{cas1}), (\ref{less2}) and (\ref{bet23}) imply that, for all $z\in D_0 \cap \overline {D}(0,3)$,
\[k_{D_0} (z^*,z)\le 1+4\pi +300h_{D'} (z^*,z)\]
and by (\ref{cas1}) and (\ref{more2}) we have that, for every $z\in D_0 \backslash \overline {D}(0,3)$,
\[k_{D_0} (z^*,z)\le 1+4\pi +\log\frac{1}{\dist (z,\partial D_0)}+300h_{D'} (z^*,z).\]
These relations in combination with (\ref{general}) give that, for every $z\in D_0$,
\[k_{D_0} (z^*,z)\le C_1\log\frac{1}{\dist (z,\partial D_0)}+C_2,\]
where  $C_2=2(1+4\pi +300C')$ and $C_1=1+600/\alpha$. So, $D_0$ is a H\"{o}lder domain.

Conversely, suppose that $D_0$ is a H\"{o}lder domain (with the classical definition with the quasi-hyperbolic distance). We fix $z^*\in D_0\backslash \D$ (see Fig.\ \ref{fig1}). By definition, there are constants $C_1>0$ and $C_2$ such that, for every $z\in D_0$,
\[k_{D_{0}}(z^*,z)\le C_1 \log \frac{\dist (z^*,\partial D_{0})}{\dist (z,\partial D_{0})}+C_2.\]
Let $z\in D'\cap \overline{\D}$ and set $w=g^{-1}(z)$ and $w^*=g^{-1}(z^*)$. Then $w\in D\backslash D(w_0,r)$ and $w^*\in\overline{D}(w_0,r)$. We have
\[2k_{D_{0}}(z^*,z)\ge  h_{D_{0}}(z^*,z)\ge  h_{D'}(z^*,z)= h_D(w^*,w);\]
see \cite[p.\ 476]{BeaPom} and \cite{Bea}. So, combining the relations above we infer that
\begin{equation}\label{ekt1}
h_D(w^*,w)\le  2C_1 \log \frac{\dist (z^*,\partial D_0)}{\dist (z,\partial D_0)}+2C_2.
\end{equation}
Since $z\in D'\cap \overline{\D}$, we have $\dist (z,\partial D_0)=\dist(z,\partial D')$. By (\ref{sphecho}) and Lemma \ref{coro} we derive that
\[\dist (z,\partial D_0)=\dist(z,\partial D')\ge \frac{1}{\pi} \dist_{\sigma}(z,\partial D')\ge \frac{c_1}{\pi} \dist_{\sigma}(w,\partial D),\]
for some constant $c_1>0$ depending on $r$ and $|w_0|$. This in conjunction with (\ref{ekt1}) and the fact that  $\dist (z^*,\partial D_0)\le |z^*|\le 4$ implies that
\begin{align}\label{sxx1}
h_D(w^*,w)&\le 2C_1\log \frac{4\pi}{c_1\dist_{\sigma}(w,\partial D)}+2C_2 \nonumber\\
&=2{C_1}\log \frac{1}{\dist_{\sigma}(w,\partial D)}+2{C_1}\log\frac{4\pi}{c_1}+2C_2,
\end{align}
for every $w\in D\backslash D(w_0,r)$.

\begin{figure}
	\begin{overpic}[width=11cm]{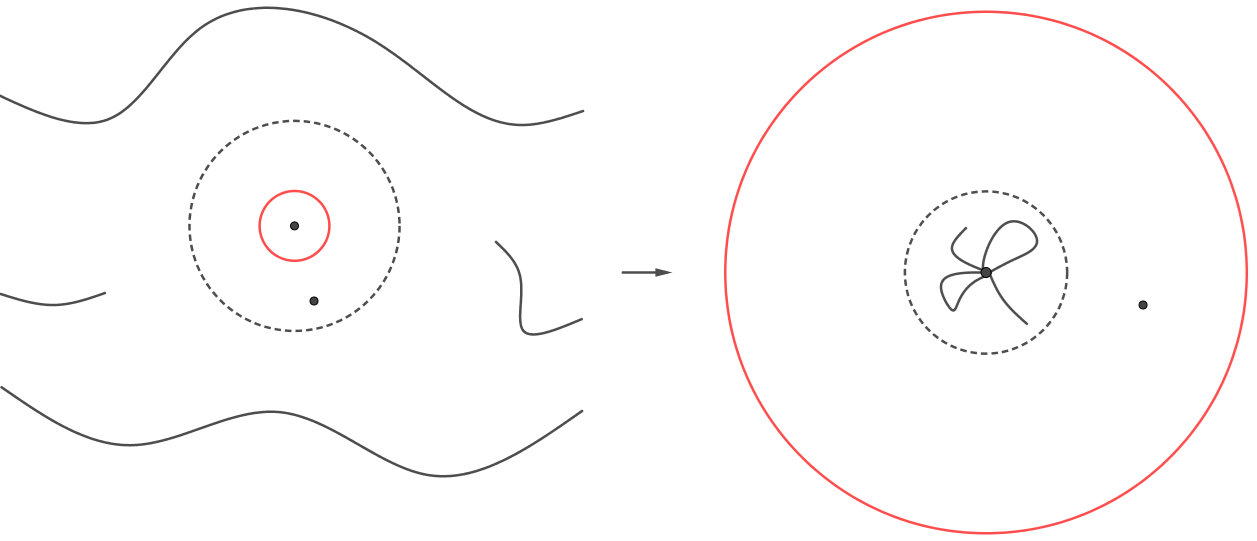}
		\put (26,19) {$w^*$}
		\put (22,26) {$w_0$}
		\put (51,24) {$g$}
		\put (77.5,18) {$0$}
		\put (92,18) {$z^*$}
		\put (8,12) {$D$}
		\put (80,3) {$D_0$}
		\put (95,1) {$D'$}
	\end{overpic}
	\caption{The domains $D,D'$ and $D_0$.}
	\label{fig1}	
\end{figure}

If $w\in D(w_0,r)$, let $\gamma$ be the line segment joining $w$ to $w^*$. Then we have
\[h_D(w^*,w)\le 2k_D(w^*,w)\le 2 \int_{\gamma}\frac{|d\zeta|}{\dist (\zeta,\partial D)}\le 4,\]
because $\dist (\zeta,\partial D)\ge r$ and $l(\gamma)\le 2r$. Moreover, by (\ref{sphecho})
\[\dist_{\sigma} (w,\partial D)\le \pi \dist ({w,\partial D})\le 3r\pi\]
and thus
\[\log \frac{1}{\dist_{\sigma} (w,\partial D)} \ge \log \frac{1}{3r\pi}.\]
Combining the relations above we infer that, for every $w\in D(w_0,r)$,
\begin{align}\label{sxx2}
h_D(w^*,w)&\le 4-\log \frac{1}{3r\pi}+\log \frac{1}{3r\pi} \nonumber\\
&\le 4+\log ({3r\pi})+ \log \frac{1}{\dist_{\sigma} (w,\partial D)}.
\end{align}

In conclusion, by (\ref{sxx1}) and (\ref{sxx2}) we derive that, for every $w\in D$,
\[h_D(w^*,w)\le C+\frac{1}{\alpha} \log \frac{1}{\dist_{\sigma} (w,\partial D)},\]
where $1/\alpha=1+2C_1$ and $C=4+\log ({3r\pi})+2{C_1}\log({4\pi}/{c_1})+2C_2$. Therefore, Theorem \ref{hyho} implies that $D$ is an unbounded H\"{o}lder domain.
\end{proof}


\begin{bibdiv}
\begin{biblist}
	
	
\bib{Bea}{article}{
	title={The hyperbolic metric and geometric function theory},
	author={A.~F. Beardon and D. Minda,},
	journal={Quasiconformal mappings and their applications},
	date={2007},
	pages={9--56}
}

	
\bib{BeaPom}{article}{
	title={The Poincar\'{e} metric of plane domains},
	author={A.~F. Beardon and Ch. Pommerenke,},
	journal={J. London Math. Soc.},
	volume={18},
	date={1978},
	number={2},
	pages={475--483}
}

\bib{BecPom}{article}{
	title={H\"{o}lder continuity of conformal mappings and non-quasiconformal Jordan curves},
	author={J. Becker and Ch. Pommerenke},
	journal={Comment. Math. Helvetici},
	volume={57},
	date={1982},
	pages={221--225}
}

\bib{Bur}{article}{
	title={Exit times of Brownian motion, harmonic majorization, and Hardy spaces},
	author={D. L. Burkholder},
	journal={Advances in Mathematics},
	volume={26},
	date={1977},
	pages={182--205}
}
\bib{Dur}{book}{
	title={Theory of $H^p$ Spaces},
	author={P.~L. Duren},
	date={1970},
	publisher={Academic Press},
	address={New York-London}
}


\bib{Gol}{book}{
	author={G.~M. Goluzin},
	title={Geometric theory of functions of a complex variable},
	series={Translations of Mathematical Monographs},
	volume={26},
	year={1969},
	publisher={American Mathematical Society, Providence, R.I.}
}

\bib{GS}{article}{
	title={Collet, {E}ckmann and {H}\"older},
	author={J. Graczyk and S. Smirnov},
	journal={Invent. Math.},
	volume={133},
	number={1},
	pages={69--96},
	year={1998},
}

\bib{Han}{article}{
	title={Hardy classes and ranges of functions},
	author={L.~J. Hansen},
	journal={Michigan Math. J.},
	volume={17},
	date={1970},
	pages={235--248}
}
\bib{Han2}{article}{
	title={The Hardy class of a spiral-like function},
	author={L.J. Hansen},
	journal={Michigan Math. J.},
	volume={18},
	date={1971},
	pages={279--282}
}
\bib{Jo}{article}{
	author={F. John},
	title={Rotation and strain},
	journal={Comm. Pure Appl. Math.},
	volume={14},
	year={1961},
	pages={391--413},
}

\bib{Kar}{article}{
	title={On the Hardy number of a domain in terms of harmonic measure and hyperbolic distance},
	author={C. Karafyllia},
	journal={Ark. Mat.},
	volume={58},
	date={2020},
	number={2},
	pages={307--331}
}
\bib{Karfin}{article}{
	title={On the Hardy number of comb domains},
	author={C. Karafyllia},
	journal={Ann. Fenn. Math.},
	volume={47},
	date={2022},
	pages={587-601}
}
\bib{Kim}{article}{
	title={Hardy spaces and unbounded quasidisks},
	author={Y.C. Kim and T. Sugawa},
	journal={Ann. Acad. Sci. Fenn. Math.},
	volume={36},
	date={2011},
	pages={291--300}
}

\bib{LR}{article}{
	author={P. Lin and S. Rohde},
	title={Conformal welding of dendrites},
	year={2018},
	pages={preprint}
}

\bib{MS}{article}{
	author={O. Martio and J. Sarvas},
	title={Injectivity theorems in the plane and space},
	journal={Ann. Acad. Sci. Fenn. Ser. A I Math.},
	number={4},
	year={1978},
	pages={383--401},
}


\bib{NP3}{article}{
author={R. N\"{a}kki and B. Palka}
title = {Lipschitz conditions, {$b$}-arcwise connectedness and conformal mappings},
journal={J. Analyse Math.},
volume={42},
year={1982/83},
pages={38--50}
}

\bib{NP2}{article}{
author={R. N\"{a}kki and B. Palka}
title = {Hyperbolic geometry and {H}\"{o}lder continuity of conformal  mappings},
journal={Ann. Acad. Sci. Fenn. Ser. A I Math.},
volume={10},
year={1985},
pages={433--444}
}

\bib{NP}{article}{
author={R. N\"{a}kki and B. Palka}
title = {Extremal length and {H}\"{o}lder continuity of conformal mappings},
journal={Comment. Math. Helv.},
volume={61},
number={3},
year={1986},
pages={389--414}
}

\bib{NV}{article}{
	author={R. N\"akki and J.\ V\"ais\"al\"a},
	title={John disks},
	journal={Expo. Math.},
	volume={9},
	number={1},
	year={1991},
	pages={3--43},
}
\bib{Co2}{article}{
	title={The Hardy class of K{\oe}nigs maps},
	author={P. Poggi-Corradini},
	journal={Michigan Math. J.},
	volume={44},
	date={1997},
	pages={495--507}
}


\bib{Pom}{book}{
	title={Boundary Behaviour of Conformal Maps},
	author={Ch. Pommerenke},
 series = {Grundlehren der mathematischen Wissenschaften [Fundamental
              Principles of Mathematical Sciences]},
    volume= {299},
 publisher = {Springer-Verlag, Berlin},
      year = {1992},
     pages = {x+300},
}

\bib{RS}{article}{
	author={S. Rohde and O. Schramm},
	title={Basic properties of {SLE}},
	journal={Ann. Math. (2)},
	volume={161},
	number={2},
	year={2005},
	pages={883--924}	
}

\bib{Smith1}{article}{
	title={H\"{o}lder domains and Poincar\'{e} domains},
	author={W. Smith and D.~A. Stegenga},
	journal={Trans. Amer. Math. Soc.},
	volume={319},
	date={1990},
	number={1},
	pages={67--100}
}
\bib{Smith2}{article}{
	title={A geometric characterization of H\"{o}lder domains},
	author={W. Smith and D.~A. Stegenga},
	journal={J. London Math. Soc.},
	volume={35},
	date={1987},
	number={2},
	pages={471--480}
}

\end{biblist}
\end{bibdiv}
\end{document}